\documentclass[reqno,12pt]{amsart}
\newcommand{\CC}{{\mathbb C}}

\def\bege{\begin{equation}} \def\ende{\end{equation}}
\def\begr{\begin{eqnarray}} \def\endr{\end{eqnarray}}
\usepackage{amsmath}
\usepackage{amssymb}
\usepackage{cite}

\usepackage{bbm}
\usepackage{amsmath}
\usepackage{txfonts}
\usepackage{stmaryrd}
\usepackage{amssymb}
\usepackage{amsfonts}
\usepackage{times}
\usepackage{mathrsfs}

\usepackage{amsthm}
\usepackage{enumerate}

\usepackage{framed}
\usepackage{lipsum}
\usepackage{color}

\def\CC{ \mathbb{C}}

\def\D{\mathbb{D}}
\def\N{\mathbb N}

\def\hD{\hat{\mathcal{D}}}
\def\dD{\mathcal{D}}

\def\vp{\varphi}
\def\om{\omega}

\def\p{{\prime}}
\def\up{\upsilon}
\def\lmd{\lambda}

\def\begr{\begin{eqnarray}} \def\endr{\end{eqnarray}}

\def\ol{\overline}
\newtheorem{Lemma}{Lemma}
\newtheorem{Theorem}[Lemma]{Theorem}

\newcounter{other}            
\newtheorem{otherth}[other]{Theorem}              

\begin{document}
\title[]{Sums of two generalized weighted  composition operators
 between Bergman spaces}

\author{ Juntao Du   and Zuoling Liu$\dagger$}

\address{Juntao Du\\ Department of mathematics, Guangdong University of Petrochemical Technology, Maoming, Guangdong, 525000,  P. R. China.}
 \email{jtdu007@163.com  }

%
\address{Zuoling Liu \\  School of Mathematics, Jiaying University, Meizhou, Guangdong, 514015,  P. R. China }
\email{zlliustu@163.com}

 \begin{abstract}
 Based on  a  judicious partitioning of   the preimage of the pseudohyperbolic disk  under the composition
symbols, in this article,  we show that the boundedness and compactness of the sum of two generalized weighted composition operators with  different  symbols and orders between Bergman spaces are properly determined by each of  individual summands.
 \thanks{$\dagger$ Corresponding author.}
 \thanks{The work was supported by 
 the Talent Recruitment Projects of GDUPT(No. 2022rcyj2008), the Science and Technology Project of Maoming City(No. 2023417), 
 the Talent Recruitment Project of Jiaying University(Grant No.131/324E1512 ) and  the Internal Research Project of Jiaying University (Grant No.131/325E0303).}
\end{abstract}
 \maketitle

\section{Introduction}
Let $H(\D)$ denote  the space of all analytic functions in  the open unit disc $\D=\{z\in\CC:|z|<1\}$. For a nonnegative function $\omega \in L^1([0,1])$, its extension to $\mathbb{D}$, defined by $\omega(z) = \omega(|z|)$ for all $z \in \mathbb{D}$, is called a radial weight. The set of doubling weights, denoted by $\hat{\mathcal{D}}$,  consists of all radial weights $\omega$ such that (see \cite{Pj2015})
 $$\hat{\omega}(r)\leq C \hat{\omega}\left(\frac{r+1}{2}\right),\,\,\,0\leq r<1$$
for some  constant $C=C(\omega)\geq1$. Here  $\hat{\om}(z)=\int_{|z|}^1\om(t)dt $.
Futhermore, if $\omega\in\hat{\mathcal{D}}$ and satisfies
\begin{align*}
\hat{\omega}(r)\geq C\hat{\omega}\left(1-\frac{1-r}{K}\right),\,\,\,0\leq r<1
\end{align*}
for some other constants $K=K(\omega)>1$ and $C=C(\omega)>1$. We refer to $\omega$ as a two-sides doubling weight and denote it by $\omega\in\dD$. For any $\lambda\in\D$,  the  Carleson square at $\lambda\in\D$ is defined by
$$S(\lambda)=\left\{re^{\mathrm{i\theta}}:|\lambda|\leq r<1, |\mbox{Arg } \lambda-\theta|<\frac{1-|\lambda|}{2}\right\}.$$
For a radial weight $\om$, let $\om(S(\lambda))=\int_{S(\lambda)}\om(z)dA(z)$.
Here and henceforth, $dA$ denotes the normalized Lebesgue area measure on $\mathbb{D}$.
Obviously, $\om(S(\lambda))\approx (1-|\lambda|)\hat{\om}(\lambda)$.
Further properties of doubling weights can be found in \cite{Pj2015, PjRj2021adv, PjRjSk2021jga} and   references therein.

For $0<p<\infty$ and a given $\omega\in\hat{\mathcal{D}}$, the Bergman space $A_\omega^p$  consists of all  functions $f\in H(\D)$ such that
$$\|f\|_{A_\omega^p}^p=\int_\mathbb{D} |f(z)|^p\omega(z)dA(z)<\infty.$$
 As usual, we write  $A_\alpha^p$ for the classical weighted Bergman space induced by the standard radial weight $\omega(z)=(\alpha+1)(1 - |z|^2)^\alpha$ with $-1<\alpha<\infty$. Throughout this paper, we assume that  $\hat{\om}(z) >0$ for all $z\in\D$. Otherwise $A_\om^p=H(\D)$.
In \cite{PjRj2021adv}, the authors characterized the Littlewood-Paley formula  on  Bergman spaces induced  by radial weights.
That is, when $\om$ is a radial weight, $0<p<\infty$ and $k\in\N$,
\begin{align}\label{0402-1}
\int_\D |f(z)|^p\om(z)dA(z)\approx \sum_{j=0}^{k-1}|f^{(j)}(0)|^p+\int_\D |f^{(k)}(z)|^p(1-|z|^2)^{kp}\om(z)dA(z)
\end{align}
holds for all $f\in H(\D)$ if and only if $\om\in\dD$.

Let  $S(\D)$ be the class of all analytic self-maps of $\D$. For $n\in\N\cup\{0\},\vp\in S(\D)$, and $u\in H(\D)$, the generalized weighted composition operator $uD_\vp^{(n)}$ is defined by
$$uD_\vp^{(n)} f=u \left(f^{(n)}\circ \vp\right), \quad f\in H(\D). $$
The operator $uD_\vp^{(n)}$, which is also called a weighted differentiation composition operator,  was introduced by Zhu in \cite{ZXL2007}.
As special cases, the generalized weighted composition operators $uD_\vp^{(n)}$ include multiplication operators $M_u$ ($n=0$ and $\vp(z)=z$),   composition operator $C_\vp$ ($n=0$ and $u(z)\equiv 1$) and weighted composition operators $uC_\vp$ ($n=0$).
Much effort has been expended on characterizing those symbols which induce bounded (compact) composition and weighted composition operators.
Readers interested in this topic can refer to \cite{CM1995, CzZr2004jlms,CzZr2007ijm,Sj1993,SjSc1990pjm,Zhu1}.
Moreover, when $u\equiv 1$, the operator $uD_\vp^{(n)}$ reduces to the differentiation-composition operator $D_\vp^{(n)}$; when $u\equiv 1$ and $\vp(z)=z$, $uD_\vp^{(n)}$ it simplifies to the $n$-th differentiation operator $D^{(n)}$.  So,  the generalized weighted composition operator   attracted a lot of attentions since it covers a lot of classical operators.

The products of $C_\vp$, $M_u$ and $D^{(n)}$ can be obtained in six ways, such as $M_uC_\vp D$, $M_uD C_\vp$, $C_\vp M_u D$, $DM_u C_\vp$, $C_\vp D M_u$ and $D C_\vp M_u$.
In order to treated those operators in a unified manner, Stevi\'c, Sharma and Bhat \cite{SsSaBa2011amc,SsSaBa2011amc-2} introduced an operator
 $$T_{u_0,u_1,\vp}=u_0D_\vp^{(0)}+u_1 D_\vp^{(1)}$$
 and characterized the boundedness, compactness and essential norm of $T_{u_0,u_1,\vp}:A_\alpha^p\to A_\alpha^p$ with some assumptions.
%
In 2020, Wang, Wang and Guo \cite{WWG2020B} introduced an natural extension of $T_{u_0,u_1,\vp}$, denoted by
$$T_{n,\vp,\vec{u}}=\sum_{k=0}^n u_kD_\vp^{(k)},$$
in which, $\vec{u}=\{u_0,u_1,...u_n\}\subset H(\D),\vp\in S(\D),$
and characterized the boundedness and compactness of the operator $T_{n,\vp,\vec{u}}$
 mapping from  admissible spaces $X$, which includes Bergman spaces, Hardy spaces and so on,  into the $k$-th weighted type spaces $\mathcal{W}_\mu^{(k)}$ and its little-oh subspces $\mathcal{W}_{\mu,0}^{(k)}$, which include the weighted type spaces, weighted Bloch spaces, weighted Zygmund spaces and their little-oh subspaces but not the Bergman spaces and Hardy spaces.
See, for example,   \cite{FG2022,WWG2020,YyLy2015caot,ZfLy2018caot} for more investigations about these operators.
In 2024, 
 the current authors  \cite{DjLsLz2024mmas}  showed that, if $1<p,q<\infty$ and $\om,\up\in\dD$,
$$\|T_{n,\vp,\vec{u}}\|_{A_\om^p\to A_\up^q}\approx \sum_{k=0}^n \|u_kD_\vp^{(k)}\|_{A_\om^p\to A_\up^q},
\quad
\|T_{n,\vp,\vec{u}}\|_{e,A_\om^p\to A_\up^q}\approx \sum_{k=0}^n \|u_kD_\vp^{(k)}\|_{e,A_\om^p\to A_\up^q}.$$
Here, $\|\cdot\|_{e,A_\om^p\to A_\up^q}$ is the essential norm of an operator from $A_\om^p$ to $A_\upsilon^q$.
More precisely,   the essential norm of a bounded operator $T$ from a Banach space $X$ to another Banach space $Y$   is defined as
$$\|T\|_{e,X\to Y}=\inf\Big\{\|T-K\|_{X\to Y};K:X\to Y \mbox{ is compact}\Big\}.$$
Obviously, $T$ is compact if and only if $\|T\|_{e,X\to Y}=0$.

In \cite{AsFt2019caot}, Acharyya and Ferguson introduced the sum of generalized weighted  composition operators $T_{n,\vec{\vp},\vec{u}}$ and generalized the operator  $T_{n,\vp,\vec{u}}$, which is defined by
 $$T_{n,\vec{\vp},\vec{u}}=\sum_{k=0}^n u_k D_{\vp_k}^{(n)},\qquad \vec{u}:=\{u_k\}_{k=0}^n\subset H(\D),\,\vec{\vp}:=\{\vp_k\}_{k=0}^n\subset S(\D),$$
 and  investigated the order-boundedness of $T_{n,\vec{\vp},\vec{u}}:A_\alpha^p\to A_\beta^q$ and the compactness of $T_{n,\vec{\vp},\vec{u}}:A_\alpha^p\to H^\infty$.
 In \cite{DjLsLz2024mmas}, we estimated the norm and essential of $T_{n,\vec{\vp},\vec{u}}:A_\om^p\to H^\infty$.
Existing results demonstrate that the properties of $T_{n,\vec{\vp},\vec{u}}$ are entirely determined by each of its individual summands.
This naturally leads to the following question:

 {\bf Q1}: what can be said about the operator $T_{n,\vec{\vp},\vec{u}}:A_\om^p\to A_\upsilon^q$?

In 2021, Choe et al. \cite{CCKY2020jfa}  established complete characterizations in terms of Carleson measures for bounded and compact differences of weighted composition operators $uC_\vp-v C_\psi$ acting on the standard weighted Bergman spaces  over the unit disk.
A principal innovation resides in the judicious partitioning of $\vp^{-1}(E_s(r))$, coupled with the strategic selection and rigorous estimation of $\left|u-v\frac{1-\ol{b}\vp}{1-\ol{b}\psi}\right|$ across each partition.
Here $E_s(a)$ is a pseudohyperbolic  disk   centered at $a\in\D$ with radius $r\in(0,1)$.
Their work yielded a significant breakthrough by resolving the open question about the compactness  for differences of  composition operators on Hardy spaces, originally posed by Shapiro and Sundberg in their seminal work \cite{SjSc1990pjm}.
Building upon this foundational methodology, subsequent research has systematically extended these results to diverse settings: to different standard weighted Bergman spaces on the unit disk \cite{CCKY2021ieot}, to multidimensional analogues in the unit ball setting \cite{CCKP2024caot}, and more recently, to weighted Bergman spaces induced by two-sides doubling weights \cite{Cj2023bkms}.

Motivated  by these researches, this short note presents an  affirmative answer to the Question Q1 when $T_{n,\vec{\vp},\vec{u}}$ only has two summands.
Unlike the reference \cite{DjLsLz2024mmas},   the proof of this conclusion is based on estimates of test functions on different subsets of the preimage of the pseudohyperbolic  disk $E_s(a)$ under the composition symbol $\vp$.

\begin{Theorem}\label{0421-1}
Suppose $1<p, q<\infty$,  $0\leq n<m<\infty$, $\om\in\dD$, $\up$ is a positive Borel measure on $\D$.
 Then, for all $\vp,\psi\in S(\D)$ and  $u_m,u_n\in H(\D)$,
$$\|u_nD_{\vp}^{(n)}+u_mD^{(m)}_{\psi}\|_{A_\om^p\to L_\up^q}\approx \|u_nD_{\vp}^{(n)}\|_{A_\om^p\to L_\up^q}+\|u_mD^{(m)}_{\psi}\|_{A_\om^p\to L_\up^q}.$$
Moreover, if $p\leq q$ and $u_nD_{\vp}^{(n)}+u_mD^{(m)}_{\psi}:A_\om^p\to L_\up^q$ is bounded,
$$\|u_nD_{\vp}^{(n)}+u_mD^{(m)}_{\psi}\|_{e,A_\om^p\to L_\up^q}
  \approx \|u_nD_{\vp}^{(n)}\|_{e,A_\om^p\to L_\up^q}+\|u_mD^{(m)}_{\psi}\|_{e,A_\om^p\to L_\up^q}.$$
\end{Theorem}

By (\ref{0624-1}) and Theorem \ref{thB} in section 2, we can obtain geometric characterizations of the norm and essential norm of  $u_nD_{\vp}^{(n)}$ and $u_mD^{(m)}_{\psi}$. Indeed, as  important  objects  in operator theory, weighted composition operators and differentiation operators on Bergman spaces have been extensively investigated, with further characterizations available in  \cite{Lb2021bams,PjRjSk2021jga} and the references therein.

Throughout this paper, the letter $C$ will represent constants, which may vary from one occurrence to another. For two positive functions $f$ and $g$, we use the notation $f \lesssim g$ to denote that there exists a positive constant $C$, independent of the arguments, such that $f \leq Cg$. Similarly, $f \approx g$ indicates that $f \lesssim g$ and $g \lesssim f$.\\

%

\section{preliminaries}

This section establishes fundamental lemmas that constitute the mathematical foundation for proving our principal theorem, which are to be systematically employed throughout the proof of the main conclusion.

The theory of pull-back measures has been instrumental in characterizing the boundedness and compactness of (generalized) weighted composition operators between Bergman spaces.
For any $u\in H(\D),\vp\in S(\D), 0<q<\infty, \up$ is a positive Borel measure on $\D$,  the pull-back measure $\mu_{u,\vp,q,\up}$ is defined by
$$\mu_{u,\vp,q,\up}(E)=\int_{\vp^{-1}(E)} |u(z)|^q d\up(z).$$
Here, $E\subset\D$ is a measurable set. Then,
\begin{align}\label{0624-1}
\|uD_\vp^{(n)} \|_{A_\om^p\to L_\up^q}= \|D^{(n)}\|_{A_\om^p\to L_{\mu_{u,\vp,q,\up}}^q}.
\end{align}
Therefore, the behaviors of $uD_\vp^{(n)}$ between Bergman spaces are determined by the embedding derivatives of Bergman spaces into Lebesgue spaces.
%
%
As usual, we write 	$\rho(a,z)=|\vp_a(z)|=\left|\frac{a-z}{1-\ol{a}z}\right|$  for the pseudohyperbolic distance between $z$
and $a$, and $E_s(a)=\{z\in\D:\rho(a,z)<s\}$ for the pseudohyperbolic disc with center $a\in\D$ and radius $s\in(0,1)$.
Moreover, a sequence $\{a_i\}_{i=1}^\infty$ is called a $\delta$-lattice   if $D=\cup_{i=1}^\infty E_\delta(a_i)$ and $E_{\delta/2}(a_i)$ are pairwise disjoint.
Then, Theorem 3 in \cite{PjRjSk2021jga}(also see \cite[Theorem 1.2]{Lb2021bams}) can be stated as follows.
\begin{otherth}\label{thB}
Let $0 < p,q<\infty, n\in\N\cup\{0\}, \om\in\dD, 0 < s < 1, \mu$ be a positive Borel
measure on $\D$. Then the following
statements hold.
\begin{enumerate}[(i)]
  \item When $0<p\leq q<\infty$, $D^{(n)}:A_\om^p\to L_\mu^q$ is bounded  if and only if $$\sup_{z\in\D}\frac{\mu(E_s(z))}{(1-|z|)^{nq}\om(S(z))^\frac{q}{p}}<\infty.$$
      Moreover,
           $$\|D^{(n)}\|_{A_\om^p\to L_\mu^q}^q\approx \sup_{z\in\D}\frac{\mu(E_s(z))}{(1-|z|)^{nq}\om(S(z))^\frac{q}{p}}.$$
  \item When $0<q<p<\infty$, the following statements are equivalent:
         \begin{enumerate}
         \item[(iia)]   $D^{(n)}:A_\om^p\to L_\mu^q$ is bounded;
         \item[(iib)]  $D^{(n)}:A_\om^p\to L_\mu^q$ is compact;
         \item[(iic)] the maximum function
           $$M_\om(\mu,n)(z):=\frac{\mu(E_s(z))}{(1-|z|)^{nq}\om(S(z))}, \,\,\,z\in\D$$
       belongs to $L_\om^\frac{p}{p-q}$.
         \end{enumerate}
       Moreover,
        $$\|D^{(n)}\|_{A_\om^p\to L_\mu^q}^q\approx \|M_\om(\mu,n)\|_{L_\om^\frac{p}{p-q}}.$$
\end{enumerate}
\end{otherth}

When $\om\in\hD$ and $0<p<\infty$, let
\begin{align*}
f_{\lambda,\gamma,\om,p}(z)=\left(\frac{1-|\lambda|^2}{1-\ol{\lambda}z}\right)^{\gamma}\frac{1}{\om(S(\lambda))^\frac{1}{p}},\, \,\lambda,z\in\D.
\end{align*}
According to \cite[Lemma 3.1]{Pj2015}, there exists $\gamma_\#=\gamma_\#(\om,p)$, whenever $\gamma>\gamma_\#$,
\begin{align*}
\|f_{\lambda,\gamma,\omega,p}\|_{A_\omega^p}\approx 1,\,\,\,\,\lambda\in\D.
\end{align*}
For brevity, we  denote $f_{\lambda,\gamma,\omega,p}$ by $f_{\lambda,\gamma}$.
Obviously, if $\gamma$ is fixed, $f_{\lmd,\gamma}$ converges to 0 uniformly on any compact subset of $\D$ as $|\lmd|\to 1$.
The  functions $\{f_{\lmd,\gamma}\}$ are widely used as test functions in investigating  the operator theory on Bergman spaces.

Lemma 4.3 in \cite{CCKY2020jfa}, which is immediate from the triangle inequality,
will be repeatedly used in the proof of the main result.
It is included here for easier references and expressed in a reorganized format.

\begin{Lemma}\label{Lemma 4.3}
Let $\varepsilon>0$. If $\eta$ and $\zeta$ are nonzero complex numbers such that $(1+\varepsilon)|\zeta|\leq |\eta|$, then
$$|\eta|+|\zeta|\geq |\eta+\zeta|\geq  \frac{\varepsilon}{2+\varepsilon}(|\eta|+|\zeta|).$$
\end{Lemma}


 When $\om\in\hD$ and $1<p<\infty$, by Theorem 7 in \cite{PjRj2021adv}, $A_\om^p$ is reflexive.
 Using Lemma 2.1 in \cite{CT2016jmaa}, the reflexivity of a Banach space induces a complete characterization of compact operators on them. Thus, we have the following lemma.
\begin{Lemma}\label{0406-2} Suppose $1<p,q<\infty$,  $\om\in\hD$, $\mu$ is a positive Borel measure on $\D$.  If $K:A_\om^p\to L_\mu^q$ is bounded,  $K$ is compact if and only if $\|K{f_n}\|_{L_\mu^q}\to 0$ as $n\to \infty$ whenever $\{f_n\}$ is bounded in $A_\om^p$ and uniformly converges to 0 on any compact subset of $\D$ as $n\to \infty$.
\end{Lemma}

We conclude this section by extending the classical essential norm estimation framework to derivative embeddings between Bergman and Lebesgue spaces.

\begin{Lemma}\label{0624-2}
Let $1< p\leq q<\infty, n\in\N\cup\{0\}, \om\in\dD, 0 < s < 1, \mu$ be a positive Borel
measure on $\D$.    If $D^{(n)}:A_\om^p\to L_\mu^q$ is bounded,
$$\|D^{(n)}\|_{e,A_\om^p\to L_\mu^q}^q\approx \limsup_{|z|\to 1}\frac{\mu(E_s(z))}{(1-|z|)^{nq}\om(S(z))^\frac{q}{p}}.$$
\end{Lemma}
\begin{proof}
For any $r\in (0,1)$, let $\chi_r$ be the characteristic function of $r\D$. That is $\chi_r(z)=1$ when $|z|<r$ and $\chi_r(z)=0$ when $|z|\geq r$. By Lemma \ref{0406-2} and Theorem \ref{thB}, $\chi_r D^{(n)}:A_\om^p\to L_\mu^q$ is compact and
$$\|D^{(n)}\|_{e,A_\om^p\to L_\mu^q}^q\leq \|D^{(n)}-\chi_r D^{(n)}\|_{A_\om^p\to L_\mu^q}\approx \sup_{z\in\D}\frac{(\mu-\mu\chi_r)(E_s(z))}{(1-|z|)^{nq}\om(S(z))^\frac{q}{p}}.$$
Letting $r\to 1$, we show
$$\|D^{(n)}\|_{e,A_\om^p\to L_\mu^q}^q\lesssim \limsup_{|z|\to 1}\frac{\mu(E_s(z))}{(1-|z|)^{nq}\om(S(z))^\frac{q}{p}}.$$

On the other hand, letting $\gamma>\gamma_\#$ be fixed, for any compact operator $K:A_\om^p\to L_\mu^q$,  when $\lmd\in\D$, we have
\begin{align*}
\|(D^{(n)}-K)f_{\lmd,\gamma}\|_{L_\mu^q}^q
&\gtrsim \int_{\D} |D^{(n)}f_{\lambda,\gamma}(z)|^qd\mu(z)-\int_\D |Kf_{\lambda,\gamma}(z)|^qd\mu(z)\\
&\geq \int_{E_s(\lambda)} |D^{(n)}f_{\lambda,\gamma}(z)|^qd\mu(z)-\int_\D |Kf_{\lambda,\gamma}(z)|^qd\mu(z).
\end{align*}
Letting $|\lambda|\to 1$, by Lemma \ref{0406-2},  we claim
$$\|D^{(n)}-K\|_{A_\om^p\to L_\mu^q}^q\gtrsim \limsup_{|\lmd|\to 1}\frac{\mu(E_s(\lmd))}{(1-|\lmd|)^{nq}\om(S(\lmd))^\frac{q}{p}}.$$
Since $K$ is arbitrary, we get the desired lower estimate. The proof is complete.
\end{proof}

\section{Proof of Theorem \ref{0421-1}}


\begin{proof}
The upper estimates of the norm and essential norm of $u_nD_{\vp}^{(n)}+u_mD^{(m)}_{\psi}$ are obvious.
Consequently, it suffices to establish the lower estimates in our analysis.
Invoking (\ref{0402-1}) with Lemmas 4 and 5 in \cite{DjLsLz2024mmas}, without loss of generality, we can assume  $0=n<m<\infty$.

For any $a\in \D\backslash\{0\}$, let  $a_N:=ae^{N(1-|a|)\mathrm{i}}$ for some $N>0$, which will be fixed later,  and
$$ \Omega_{s}(a):=\left\{w\in\D:|1-\ol{a}w|<4\sup_{z\in E_s(a) }|1-\ol{a}z|\right\}.$$
Then, $\sup\limits_{a\in\D}\rho(a,a_N)<1$.
By Lemma 4.30 in\cite{Zhu1}, there exists a constant $C_0=C_0(s)$,  for any $a,z\in\D$ and $w\in\Omega_{s}(a)$, we have
\begin{align}\label{0401-1}
|1-\ol{a}w|<C_0(1-|a|)
\end{align}
and
\begin{align}\label{0619-1}
|1-\ol{a_N}z|\approx |1-\ol{a}z|.
\end{align}
Letting $N=4C_0$, as $|a|$ approaches 1, we have $2C_0\leq \frac{|e^{N(1-|a|)\mathrm{i}}-1|}{1-|a|}\leq 6C_0 $.
By $E_s(a)\subset \Omega_s(a)$ and (\ref{0401-1}),
\begin{align}\label{4.9}
\left|\frac{1-\ol{a_{N}}z}{1-\ol{a_{N}}w}\right|
\leq \frac{|e^{N(1-|a|)\mathrm{i}}-1|+|1-\ol{a}z|}{\big||e^{N(1-|a|)\mathrm{i}}-1|-|1-\ol{a}w|\big|}
\leq 7,\quad \mbox{ when }z\in E_s(a), w\in \Omega_s(a).
\end{align}
Meanwhile, by the definition of the set $\Omega_{s}(a)$, 
we have
\begin{align}\label{0622-7}
\left|\frac{1-\ol{a}z}{1-\ol{a}w}\right|\leq \frac{1}{4},\quad \mbox{ when }z\in E_s(a), w\not\in \Omega_s(a).
\end{align}

For brief, put
$$T=u_0C_{\vp}+u_mD^{(m)}_{\psi},\,\quad \|\cdot\|_{A_\om^p\to L_\up^q}=\|\cdot\|,\quad \|\cdot\|_{e,A_\om^p\to L_\up^q}=\|\cdot\|_e, $$
and
$$Q_b=\frac{1-\ol{b}\vp}{1-\ol{b}\psi},\quad  (x)_m=x(x+1)\cdots(x+m-1),\quad v_x=\frac{(2x)_m}{(x)_m}.$$

Let $\gamma>\gamma_\#$ be fixed.
Since $v_\gamma>1$, we can choose $\varepsilon>0$  small enough such that $\eta:=\frac{1+\varepsilon}{v_\gamma}<1$.
Then, we construct a decomposition of the set $\vp^{-1}( E_s(a) )$ in the following manner:
\begin{align*}
G_1(a)&=\left\{z\in F_1(a):\frac{|v_\gamma(1-Q_{a_N}^\gamma)|}{|1-v_\gamma|}>1+\varepsilon\right\},\\
G_2(a)&=\left\{z\in F_1(a):\frac{|v_\gamma(1-Q_{a_N}^\gamma)|}{|1-v_\gamma|}\leq 1+\varepsilon, \, |Q_{a_N}^\gamma|>\frac{1+\eta}{2}\right\},\\
G_3(a)&=\left\{z\in F_1(a):\frac{|v_\gamma(1-Q_{a_N}^\gamma)|}{|1-v_\gamma|}\leq 1+\varepsilon, \, |Q_{a_N}^\gamma|\leq \frac{1+\eta}{2}\right\},\\
G_4(a)&=\vp^{-1}( E_s(a) )\backslash F_1(a).
\end{align*}
Here,
 \begin{align*}
 F_1(a):=\vp^{-1}( E_s(a) )\cap \psi^{-1}( \Omega_{s}(a)).
 \end{align*}
Therefore, by (\ref{4.9}), we find
\begin{align}\label{0605-1}
|Q_{a_N}(z)|\leq 7 \quad\mbox{ when }\quad z\in F_1(a).
\end{align}
The proof will be completed through three sequential parts:

{\bf Part (a): $1<p\leq q<\infty$ and the lower estimate of $\|T\|$.}

For any given $k\geq 1$, as $|a|$ approaches 1,    by (\ref{0401-1}), (\ref{0619-1}) and (\ref{0605-1}), we get
\begin{align*}
\int_{F_1(a)} \left|Tf_{a_N,k\gamma}\right|^q d\up
&\gtrsim \frac{1}{\om(S(a))^\frac{q}{p}}\int_{F_1(a)} \left|u_0+
\frac{(k\gamma)_m\ol{a_N}^m u_m }{(1-\ol{a_N}\psi)^{m}}Q_{a_N}^{k\gamma}\right|^q d\up  \\
& \gtrsim \frac{1}{\om(S(a))^\frac{q}{p}}\int_{F_1(a)} \left|u_0\frac{(2k\gamma)_m}{(k\gamma)_m} Q_{a_N}^{k\gamma}+
\frac{(2k\gamma)_m\ol{a_N}^m u_m }{(1-\ol{a_N}\psi)^{m}}Q_{a_N}^{2k\gamma}\right|^q d\up
\end{align*}
and
\begin{align*}
\int_{F_1(a)} \left|Tf_{a_N,2k\gamma}\right|^q d\up
&\gtrsim\frac{1}{\om(S(a))^\frac{q}{p}}\int_{F_1(a)} \left| u_0+
\frac{(2k\gamma)_m\ol{a_N}^m u_m }{(1-\ol{a_N}\psi)^{m}}Q_{a_N}^{2k\gamma}\right|^q d\up.
\end{align*}
Then, the Triangle inequality implies
\begin{align}
\int_{F_1(a)} \left(\left|Tf_{a_N,k\gamma}\right|^q +\left|Tf_{a_N,2k\gamma}\right|^q\right)d\up
&\gtrsim \frac{1}{\om(S(a))^\frac{q}{p}}\int_{F_1(a)}\left|
1-v_{k\gamma} Q_{a_N}^{k\gamma}
\right|^q
 |u_0|^q d\up.
\label{0619-2}
\end{align}

On $G_1(a)$, by Lemma \ref{Lemma 4.3}, we get
\begin{align*}
\left|1-v_\gamma Q_{a_N}^\gamma\right|
&=\left|1-v_\gamma+v_\gamma-v_\gamma Q_{a_N}^\gamma\right|
\approx \left|1-v_\gamma\right|+\left|v_\gamma-v_\gamma Q_{a_N}^\gamma\right|\geq v_\gamma-1.
\end{align*}

Similarly, on $G_2(a)$, we have
$$
\frac{|1-Q_{a_N}^\gamma|}{|(1-v_\gamma) Q_{a_N}^\gamma|}
= \frac{|v_\gamma(1-Q_{a_N}^\gamma)|}{|1-v_\gamma|}\frac{1}{v_\gamma|Q_{a_N}^\gamma|}
< \frac{2(1+\varepsilon)}{v_\gamma(1+\eta)}
<1,
$$
and
\begin{align*}
\left|1-v_\gamma Q_{a_N}^\gamma\right|
&=\left|1-Q_{a_N}^\gamma+Q_{a_N}^\gamma-v_\gamma Q_{a_N}^\gamma\right|\\
&\approx \left|1-Q_{a_N}^\gamma\right|+\left|Q_{a_N}^\gamma-v_\gamma Q_{a_N}^\gamma\right|
\geq \frac{(v_\gamma-1)(\eta+1)}{2}.
\end{align*}

On $G_3(a)$, since $|Q_{a_N}^\gamma|\leq \frac{\eta+1}{2}<1$ and $\lim\limits_{k\to\infty }v_{k\gamma}=2^m$,
 there exists  $k_0>1$ such that $v_{k_0\gamma}|Q_{a_N}^{k_0\gamma}|<\frac{1}{2}$, which implies
\begin{align*}
\left|1-v_{k_0\gamma} Q_{a_N}^{k_0\gamma}\right|
&\geq \frac{1}{2}.
\end{align*}
Therefore, we obtain
\begin{align*}
\|T\|&\gtrsim \|Tf_{a_N,\gamma}\|_{L_\up^q}^q+\|Tf_{a_N,2\gamma}\|_{L_\up^q}^q+\|Tf_{a_N,k_0\gamma}\|_{L_\up^q}^q+\|Tf_{a_N,2k_0\gamma}\|_{L_\up^q}^q\\
&\geq  \int_{G_1(a)\cup G_2(a)} \left(\left|Tf_{a_N,\gamma}\right|^q +\left|Tf_{a_N,2\gamma}\right|^q \right)d\up
+\int_{G_3(a)} \left(\left|Tf_{a_N,k_0\gamma}\right|^q +\left|Tf_{a_N,2k_0\gamma}\right|^q\right) d\up  \nonumber\\
&\gtrsim
\frac{1}{\om(S(a))^\frac{q}{p}}\left(
\int_{G_1(a)\cup G_2(a)}\left| 1-v_{\gamma} Q_{a_N}^\gamma \right|^q  |u_0|^q d\up
+\int_{G_3(a)}\left| 1-v_{k_0\gamma} Q_{a_N}^{k_0\gamma} \right|^q  |u_0|^q d\up\right) \\
&\gtrsim
\frac{1}{\om(S(a))^\frac{q}{p}}\left(
\int_{F_{1}(a)} |u_0|^q d\up\right).
\end{align*}

On $G_4(a)$, by (\ref{0622-7}), $|Q_a|<\frac{1}{4}$.
So, there exists $k_1>1$ such that $|v_{k_1\gamma}Q_a^{k_1\gamma}|<\frac{1}{2}$.
Then, as established in (\ref{0619-2}), we show
\begin{align*}
\|T\|&\gtrsim \|Tf_{a,k_1\gamma}\|_{L_\up^q}^q+\|Tf_{a,2k_1\gamma}\|_{L_\up^q}^q \\
&\geq  \int_{G_4(a)} \left(\left|Tf_{a,k_1\gamma}\right|^q +\left|Tf_{a,2k_1\gamma}\right|^q \right)d\up\\
&\gtrsim \frac{1}{\om(S(a))^\frac{q}{p}}\int_{G_4(a)}\left|
u-v_{k_1\gamma} Q_a^{k_1\gamma} \right|^q  |u_0|^q d\up
\gtrsim \frac{1}{\om(S(a))^\frac{q}{p}}\int_{G_4(a)}  |u_0|^q d\up .
\end{align*}
Therefore, as $|a|\to 1$, we  infer that
\begin{align}\label{0625-1}
\|T\|^q\gtrsim \frac{1}{\om(S(a))^\frac{q}{p}}\int_{\vp^{-1}( E_s(a) )} |u_0|^q d\up.
\end{align}
So, we can choose a constant $\tau\in(0,1)$ such that (\ref{0625-1}) holds for all $\tau<|a|<1$.
When $|a|\leq \tau$, letting $h\equiv 1$,
\begin{align}\label{0624-3}
\|T\|^q
\gtrsim \|Th\|_{L_\up^q}^q
\gtrsim \frac{1}{\om(S(a))^\frac{q}{p}}\int_{\vp^{-1}( E_s(a) )} |u_0|^q d\up.
\end{align}
Therefore, by (\ref{0624-1}) and Theorem \ref{thB}, $$\|u_0C_\vp\|\lesssim \|T\|.$$
Thus, $\|u_mD_\psi^{(m)}\|\lesssim \|T\|.$

{\bf Part (b): $1<p\leq q<\infty$ and the lower estimate of $\|T\|_e$.}

In the previous process, in order to get (\ref{0625-1}), we decompose $\vp^{-1}(E_s(a))$ into four disjoint subsets $G_j(1\leq j\leq 4)$.
On every region $G_j(a)$, we  choose $s_j\in\{1,k_0,k_1\}$ and $\lambda_j\in\{a_N,a\}$ properly to obtain
\begin{align}\label{0330-1}
\|Tf_{\lambda_j,s_j\gamma}\|_{A_\upsilon^q}^q+\|Tf_{\lambda_j,2s_j\gamma}\|_{A_\upsilon^q}^q
&\geq\int_{G_j(a)} \left(\left|Tf_{\lambda_j,s_j\gamma}\right|^q +\left|Tf_{\lambda_j,2s_j\gamma}\right|^q \right)d\up\nonumber\\
&\gtrsim \frac{1}{\om(S(a))^\frac{q}{p}}\int_{G_j(a)}    |1-v_{s_j\gamma} Q_{\lambda_j}^{s_j\gamma}| |u_0|^qd\upsilon  \nonumber\\
&\gtrsim \frac{1}{\om(S(a))^\frac{q}{p}} \int_{G_j(a)} |u_0|^q d\up.
\end{align}
Since $\cup_{j=1}^4 G_j(a)=\vp^{-1}(E_s(a))$, whenever $K:A_\om^p\to L_\up^q$ is compact, we deduce that
\begin{align*}
\|T-K\|^q
\gtrsim&  \sum_{1\leq j\leq 4}\sum_{k=1,2} \|(T-K)f_{\lambda_j,ks_j\gamma}\|_{L_\up^q}^q \\
\gtrsim&  \sum_{1\leq j\leq 4}\sum_{k=1,2} \left(\|Tf_{\lambda_j,ks_j\gamma}\|_{L_\up^q}^q
-\|Kf_{\lambda_j,ks_j\gamma}\|_{L_\up^q}^q\right)\\
\gtrsim& \frac{1}{\om(S(a))^\frac{q}{p}}\int_{\vp^{-1}( E_s(a) )}  |u_0|^q d\up
-\sum_{1\leq j\leq 4}\sum_{k=1,2} \|Kf_{\lambda_j,ks_j\gamma}\|_{L_\up^q}^q.
\end{align*}
Since $K$ is arbitrary and compact, letting $|a|\to 1$, by Lemma \ref{0406-2},  we have
$$\|T\|_e^q \gtrsim
\limsup_{|a|\to 1}\frac{1}{\om(S(a))^\frac{q}{p}}\int_{\vp^{-1}( E_s(a) )} |u_0|^q d\up .
$$
Then,  (\ref{0624-1}) and Lemma \ref{0624-2} deduce
$$\|T\|_e\gtrsim \|u_0C_\vp\|_e .$$
Therefore, $$\|T\|_e\gtrsim \|u_m D^{(m)}_\psi\|_e .$$

{ \bf Part (c): $1<q<p<\infty$ and the lower estimates of $\|T\|$.}

 Let $\{r_k(t)\}$  be  Rademacher functions, $\vec{a}=\{a_k\}_{k=1}^\infty$ be a $\frac{s}{2}$-lattice in $\D$
 and  $\vec{\lambda}=\{\lambda_k\}$ be a sequence given by one of $\vec{a}:=\{a_k\}$ and $\overrightarrow{a_N}:=\{a_{k,N}\}$.
 Here, $a_{k,N}=e^{N(1-|a_k|)\mathrm{i}}a_k$.
 We claim that $\{a_{k,N}\}_{k=1}^\infty$ is separated. Otherwise, for any $0<x<\frac{1}{3}$, there exist $a_{i,N}$ and $a_{j,N}$ such that
 $E_x(a_{i,N})\cap E_x(a_{j,N})\neq \O$. By \cite[Lemma 4.3]{CCKY2021ieot}, $\rho(a_i,a_j)<2(1+8N)x$.
 This is contradictory to $\vec{a}$ is a lattice.
Moreover, we can assume $\inf|a_k|>0$ and $\{|a_k|\}$ is  increasing.
Therefore, there exists a $M>0$ such that $|a_k|>\tau$ if and only if $k>M$.
Here, the constant $\tau$ is that decided by (\ref{0625-1}).

By \cite[Thoerem 1]{PjRjSk2021jga}, if $y\geq \gamma_\#$ is fixed,   for any $\vec{c}=\{c_k\}_{k=1}^\infty\in l^p$,
$\|g_{\vec{c},y,t,\vec{\lambda}}\|_{A_\om^p}\lesssim \|\vec{c}\|_{l^p}$,  in which
$$g_{\vec{c},y,t,\vec{\lambda}}(z)=\sum_{k=1}^\infty c_k r_k(t)f_{\lambda_k,y}(z).  $$
Let $\chi_k$ be the  characteristic function of $E_s(a_k)$.
By Fubini's theorem,  Khinchin's inequality  and (\ref{0619-1}),  we conclude that
 \begin{align*}
\int_0^1 \|T  g_{\vec{c},y,t,\vec{\lambda}}\|_{L_\up^q}^qdt
=&\int_\D\int_0^1 \left|\sum_{k=1}^\infty c_k r_k(t) T f_{\lambda_k,y}  \right|^q   dt   d\up \\
\approx& \int_\D\left( \sum_{k=1}^\infty |c_k|^2|T  f_{\lambda_k,y}|^2\right)^\frac{q}{2}  d\up
\geq \int_\D\left( \sum_{k=1}^\infty |c_k|^2|T  f_{\lambda_k,y}|^2(\chi_k\circ\vp)\right)^\frac{q}{2}  d\up  \\
\approx & \sum_{k=1}^\infty |c_k|^q \int_{\vp^{-1}(E_s(a_k))} |T  f_{\lambda_k,y}|^qd\up
= \sum_{j=1}^4 \sum_{k=1}^\infty |c_k|^q \int_{G_j(a_k)} |T  f_{\lambda_k,y}|^qd\up .
\end{align*}
Analogous to the parametric derivation leading to (\ref{0330-1}), when $j=1,2,3,4$, we can choose $\vec{\lambda}:=\{\lambda_{j,k}\}\in\{\vec{a},\overrightarrow{a_N}\}$ and $s_j\in\{1,k_0,k_1\}$  such that
\begin{align*}
\int_{G_j(a_k)} \left(\left|Tf_{\lambda_{j,k},s_j\gamma}\right|^q +\left|Tf_{\lambda_{j,k},2s_j\gamma}\right|^q \right)d\up
&\gtrsim \frac{1}{\om(S(a_k))^\frac{q}{p}}\int_{G_j(a_k)}    |1-v_{s_j\gamma} Q_{\lambda_{j,k}}^{s_j\gamma}| |u_0|^qd\upsilon  \nonumber\\
&\gtrsim \frac{1}{\om(S(a_k))^\frac{q}{p}} \int_{G_j(a_k)} |u_0|^q d\up
\end{align*}
as $|a_{k}|>\tau$.
Therefore, when $j=1,2,3,4$,
 \begin{align*}
\int_0^1 \|T  g_{\vec{c},s_j\gamma,t,\vec{\lambda}}\|_{L_\up^q}^qdt +  \int_0^1 \|T  g_{\vec{c},2s_j\gamma,t,\vec{\lambda}}\|_{L_\up^q}^qdt
\gtrsim&  \sum_{k=M}^\infty
\frac{|c_k|^q}{\om(S(a_k))^\frac{q}{p}}\int_{G_j(a_k)}\left|u_0\right|^q  d\up.
\end{align*}
Since $\cup_{j=1}^4 G_j(a_k)=\vp^{-1}(a_k)$ and
$$\|T  g_{\vec{c},s_j\gamma,t,\vec{\lambda}}\|_{L_\up^q}
\leq \|T\|\,\|g_{\vec{c},s_j\gamma,t,\vec{\lambda}}\|_{A_\om^p}
\lesssim \|T\|\,\|\vec{c}\|_{l^p},$$
we obtain
\begin{align*}
\|T\|^q\|\vec{c}\|_{l^p}^q\gtrsim
\sum_{k=M}^\infty \frac{ |c_k|^q}{\om(S(a_k))^\frac{q}{p}}
\int_{\vp^{-1}(E_s(a_k))}  |u_0|^qd\up.
\end{align*}
By (\ref{0624-3}), it follows that
\begin{align*}
\|T\|^q\|\vec{c}\|_{l^p}^q
\gtrsim&  \sum_{k=1}^\infty
\frac{|c_k|^q}{\om(S(a_k))^\frac{q}{p}}\int_{\vp^{-1}(E_s(a_k))}|u_0|^q  d\up.
\end{align*}
Let $\xi_k=\frac{1}{\om(S(a_k))^\frac{q}{p}}\int_{\vp^{-1}(E_s(a_k))}|u_0|^q  d\up$.
Since $\{|c_k|^q\}\in l^\frac{p}{q}$ and $\|\{|c_k|^q\}\|_{l^\frac{p}{q}}=\|\vec{c}\|_{l^p}^q$,
we have $\{\xi_k\}\in (l^\frac{p}{q})^*\simeq l^\frac{p}{p-q}$ and $\|\{\xi_k\}\|_{l^\frac{p}{p-q}}\lesssim \|T\|^q$.
Let the pull-back measure $\mu$ on $\D$ be defined by
$$\mu(E)=\int_{\vp^{-1}(E)}|u_0|^qd\up.$$
It is deduced that
\begin{align*}
\int_\D \left(\frac{\mu(E_\frac{s}{2}(z))}{\om(S(z))}\right)^\frac{p}{p-q}\om(z)dA(z)
&\leq \sum_{k=1}^\infty \int_{E_\frac{s}{2}(a_k)} \left(\frac{\mu(E_\frac{s}{2}(z))}{\om(S(z))}\right)^\frac{p}{p-q}\om(z)dA(z)\\
&\lesssim \sum_{k=1}^\infty \left(\frac{\mu(E_s(a_k))}{\om(S(a_k))}\right)^\frac{p}{p-q} \om(E_{\frac{s}{2}}(a_k))\\
&\lesssim \sum_{k=1}^\infty \left(\frac{\mu(E_s(a_k))}{\om(S(a_k))^\frac{q}{p}}\right)^\frac{p}{p-q} \\
&\lesssim \|T\|^\frac{pq}{p-q}.
\end{align*}
By (\ref{0624-1}) and Theorem \ref{thB}, it can be verified that
$$\|u_0C_\vp\|\approx \|I_d\|_{A_\om^p\to L_\mu^q}
\approx \left(\int_\D \left(\frac{\mu(E_\frac{s}{2}(z))}{\om(S(z))}\right)^\frac{p}{p-q}\om(z)dA(z)\right)^\frac{p-q}{pq}\lesssim \|T\|.$$
Consequently, we also have
$$\|u_m D_\psi^{(m)}\|\lesssim \|T\|.$$
Moreover, by Theorem \ref{thB},   the boundedness of $T:A_\om^p\to L_\up^q(q<p)$ implies the compactness of it.
The proof is complete.
\end{proof}

\end{document}